\documentclass[reqno]{amsart}
\usepackage{amssymb}
\usepackage{url}

\newtheorem{thm}{Theorem}
\newtheorem{cor}{Corollary}
\newtheorem{lem}{Lemma}

\theoremstyle{definition}

\theoremstyle{remark}

\numberwithin{equation}{section}
\begin{document}

\title[A New Super Congruence Involving Multiple Harmonic Sums]
{A New Super Congruence Involving \\ Multiple Harmonic Sums}
\author{ LIUQUAN WANG }

\address{Department of Mathematics, National University of Singapore, Singapore, 119076, Singapore}
\email{wangliuquan@nus.edu.sg; mathlqwang@163.com}

\subjclass[2010]{Primary 11A07, 11A41.}

\keywords{super congruences, Bernoulli numbers, harmonic sums}

\date{June 1 , 2014.}

\dedicatory{}

% -----------------------------------------------------------

\begin{abstract}
Let ${\mathcal{P}_{n}}$ denote the set of positive integers which are prime to  $n$. Let $B_{n}$ be the $n$-th Bernoulli number. For any prime $p >5$ and integer $r\ge 2$, we prove that
\begin{displaymath}
\sum\limits_{\begin{smallmatrix}
 {{l}_{1}}+{{l}_{2}}+\cdots +{{l}_{5}}={{p}^{r}} \\
 {{l}_{1}},\cdots ,{{l}_{5}}\in {\mathcal{P}_{p}}
\end{smallmatrix}}{\frac{1}{{{l}_{1}}{{l}_{2}}{{l}_{3}}{{l}_{4}}{{l}_{5}}}}\equiv -\frac{5!}{6}{{p}^{r-1}}{{B}_{p-5}} \pmod{{{p}^{r}}}.
\end{displaymath}
This gives an extension of a family of super congruences found by Wang, Cai and Zhao.
\end{abstract}

% -----------------------------------------------------------
\maketitle
% -----------------------------------------------------------

\section{Introduction}

Zhao \cite{Zhao1} first discovered a curious congruence
\begin{equation}\label{zhao1}
\sum\limits_{\begin{smallmatrix}
 i+j+k=p \\
 i,j,k>0
\end{smallmatrix}}^{{}}{\frac{1}{ijk}\equiv -2{{B}_{p-3}} \pmod {p}.}
\end{equation}
Here $p\ge 3$ is a prime and $B_{n}$ is the $n$-th Bernoulli number, which is defined by
\[\frac{x}{{{e}^{x}}-1}=\sum\limits_{n=0}^{\infty }{\frac{{{B}_{n}}}{n}{{x}^{n}}}.\]

Zhou and Cai \cite{zhouxia} generalized this congruence to the case of arbitrary number of variables. They showed that for any prime $p \ge 5$ and positive integer $n \le p-2$,
\begin{equation}\label{zhou}
\sum\limits_{\begin{smallmatrix}
 {{l}_{1}}+{{l}_{2}}+\cdots +{{l}_{n}}=p, \\
      {{l}_{1}},{{l}_{2}},\cdots ,{{l}_{n}}>0
\end{smallmatrix}}{\frac{1}{{{l}_{1}}{{l}_{2}}\cdots {{l}_{n}}}}\equiv \left\{ \begin{array}{ll}
   -(n-1)!{{B}_{p-n}} \pmod{p}, & \textrm{if $2 \nmid n$;} \\
 -\frac{n}{2(n+1)}n!{{B}_{p-n-1}}p  \pmod{{{p}^{2}}}, & \textrm{if $2 | n$.}
\end{array} \right.
\end{equation}
In another direction, Wang and Cai \cite{Wang} gave a new generalization of (\ref{zhao1}) by replacing prime $p$ to any prime power.
Let ${\mathcal{P}_{n}}$ denote the set of positive integers which are prime to  $n$. They proved that for any prime $p \ge 3$,
\begin{equation}\label{Wang}
\sum\limits_{\begin{smallmatrix}
 i+j+k={{p}^{r }} \\
 i,j,k\in {\mathcal{P}_{p}}
\end{smallmatrix}}^{{}}{\frac{1}{ijk}\equiv -2{{p}^{r-1}}{{B}_{p-3}} \pmod{p^r}}.
\end{equation}

Zhao \cite{Zhao4} extended this result to the case when there are four variables. He proved that for and prime $p\ge 5$ and integer $r\ge 2$,
\begin{equation}\label{Zhao2}
\sum\limits_{\begin{smallmatrix}
 {{l}_{1}}+\cdots +{{l}_{4}}={{p}^{r}} \\
 {{l}_{1}},\cdots ,{{l}_{4}}\in {\mathcal{P}_{p}}
\end{smallmatrix}}{\frac{1}{{{l}_{1}}{{l}_{2}}{{l}_{3}}{{l}_{4}}}}\equiv -\frac{4!}{5}{{p}^{r}}{{B}_{p-5}} \pmod{{p}^{r+1}}.
\end{equation}

In viewing of congruences (\ref{Wang}) and (\ref{Zhao2}), for $1 \le k \le n$ we define
\[S_{n}^{(k)}({{p}^{r}})=\sum\limits_{\begin{smallmatrix}
 {{l}_{1}}+\cdots +{{l}_{n}}=k{{p}^{r}} \\
 l_{i} < p^{r}, l_{i} \in {\mathcal{P}_{p}}
\end{smallmatrix}}{\frac{1}{{{l}_{1}}{{l}_{2}}\cdots {{l}_{n}}}}.\]

It would be very interesting to find a general congruence for $S_{n}^{(1)}(p^{r})$ modulo $p^r$ when $n$ is odd or modulo $p^{r+1}$ when $n$ is even. The cases $n=3$ and $4$ have been solved from (\ref{Wang}) and (\ref{Zhao2}). As $n$ increases, the problem becomes much more difficult. The main goal of this paper is to establish the result for the case $n=5$.
\begin{thm}\label{main}
Let $p > 5$ be a prime and $r\ge 2$ be an integer. We have
\[\sum\limits_{\begin{smallmatrix}
 {{l}_{1}}+{{l}_{2}}+\cdots +{{l}_{5}}={{p}^{r}} \\
 {{l}_{1}},\cdots ,{{l}_{5}}\in {\mathcal{P}_{p}}
\end{smallmatrix}}{\frac{1}{{{l}_{1}}{{l}_{2}}{{l}_{3}}{{l}_{4}}{{l}_{5}}}}\equiv -\frac{5!}{6}{{p}^{r-1}}{{B}_{p-5}} \pmod{{{p}^{r}}}.\]
\end{thm}
While $r=1$, by (\ref{zhou}) the congruence should be ${{S}_{5}^{(1)}}(p)\equiv -4!{{B}_{p-5}}$ (mod $p$). This is different to the case $r\ge 2$. From (\ref{zhou}) and (\ref{Zhao2}), this phenomenon also happens for the case $n=4$.
\begin{thm}\label{thm2}
Let $p > 5$ be a prime and $n$ be a positive integer. Suppose  $p^{r}|n$ but $p^{r+1} \nmid n$ for some positive integer $r$.\\
 (i) If $r=1$, then
\[\sum\limits_{\begin{smallmatrix}
 {{l}_{1}}+{{l}_{2}}+\cdots +{{l}_{5}}=n \\
 {{l}_{1}},\cdots ,{{l}_{5}}\in {\mathcal{P}_{p}}
\end{smallmatrix}}{\frac{1}{{{l}_{1}}{{l}_{2}}{{l}_{3}}{{l}_{4}}{{l}_{5}}}}\equiv -\frac{4!}{6}\cdot\Big(5\cdot\frac{n}{p}+\big(\frac{n}{p}\big)^{3}\Big){{B}_{p-5}} \pmod{p}.\]
(ii) If $r \ge 2$, then
\[\sum\limits_{\begin{smallmatrix}
 {{l}_{1}}+{{l}_{2}}+\cdots +{{l}_{5}}=n \\
 {{l}_{1}},\cdots ,{{l}_{5}}\in {\mathcal{P}_{p}}
\end{smallmatrix}}{\frac{1}{{{l}_{1}}{{l}_{2}}{{l}_{3}}{{l}_{4}}{{l}_{5}}}}\equiv -\frac{5!}{6}\cdot\frac{n}{p}{{B}_{p-5}} \pmod{{{p}^{r}}}.\]
\end{thm}
In particular, if $n=p^{r}$, then part (ii) of Theorem \ref{thm2} becomes Theorem \ref{main}.

Different from the method used for proving (\ref{Wang})-(\ref{Zhao2}) (see \cite{Wang,Zhao4}), the idea to prove Theorem \ref{main} is to establish the recurrence relation
\[S_{5}^{(1)}({{p}^{r+1}})\equiv pS_{5}^{(1)}({{p}^{r}}) \pmod {{{p}^{r+1}}}, \quad r \ge 2.\]
And then we only need to prove that $S_{5}^{(1)}(p^{2}) \equiv -\frac{5!}{6}pB_{p-5}$ (mod $p^{2}$). This idea can also be applied to give a new proof of (\ref{Wang}) and (\ref{Zhao2}).

% -----------------------------------------------------------
\section{Preliminaries}

\begin{lem}\label{rec}
Let $p >n $ be a prime and $1 \le k \le n-1$. For any integer $r\ge 1$, we have\\
(i) $S_{n}^{(k)}({{p}^{r}})\equiv {{(-1)}^{n}}S_{n}^{(n-k)}({{p}^{r}})$ \text{\rm{(mod ${{p}^{r}}$)}}. \\
(ii) ${{S}_{n}^{(1)}}({{p}^{r+1}})\equiv \sum\limits_{k=1}^{n-1} {\binom {p-k+n-1}{n-1} S_{n}^{(k)}({{p}^{r}})} $  \text{\rm{(mod  ${{p}^{r+1}})$}}.
\end{lem}

\begin{proof}
(i) Since $({{y}_{1}},\cdots ,{{y}_{n}})\leftrightarrow ({{p}^{r}}-{{y}_{1}},\cdots ,{{p}^{r}}-{{y}_{n}})$ gives a bijection between the solutions of \[{{y}_{1}}+\cdots +{{y}_{n}}=k{{p}^{r}},\quad y_{i} \in \mathcal{P}_{p},\quad  y_{i}<p^{r}, \quad 1 \le i \le n\]
and
\[{{y}_{1}}+\cdots +{{y}_{n}}=(n-k){{p}^{r}}, \quad y_{i} \in \mathcal{P}_{p}, \quad y_{i}<p^{r}, \quad 1 \le i \le n\]
we have
\begin{displaymath}
\begin{split}
  S_{n}^{(k)}(p^{r}) & = \sum\limits_{\begin{smallmatrix}
 {{y}_{1}}+\cdots +{{y}_{n}}=k{{p}^{r}} \\
 y_{i} < p^r, y_{i}\in {\mathcal{P}_{p}}
\end{smallmatrix}}{\frac{1}{{{y}_{1}}\cdots {{y}_{n}}}} \\
&=\sum\limits_{\begin{smallmatrix}
 {{y}_{1}}+\cdots +{{y}_{n}}=(n-k){{p}^{r}} \\
  y_{i} < p^r, y_{i}\in {\mathcal{P}_{p}}
\end{smallmatrix}}{\frac{1}{({{p}^{r}}-{{y}_{1}})\cdots ({{p}^{r}}-{{y}_{n}})}} \\
 & \equiv {{(-1)}^{n}}S_{n}^{(n-k)}({{p}^{r}}) \pmod {{{p}^{r}}}. \\
\end{split}
\end{displaymath}
(ii) For any $n$-tuple $({{l}_{1}},\cdots ,{{l}_{n}})$  of positive integers satisfying  ${{l}_{1}}+{{l}_{2}}+\cdots +{{l}_{n}}={{p}^{r+1}}$ and  ${{l}_{1}},\cdots ,{{l}_{n}}\in {\mathcal{P}_{p}}$, we  rewrite them as
\[{{l}_{i}}={{x}_{i}}{{p}^{r}}+{{y}_{i}}, \quad 0\le {{x}_{i}}<p, \quad 1\le {{y}_{i}}<{{p}^{r}}, \quad y_{i}\in {\mathcal{P}_{p}}, \quad 1\le i \le n.\]
 Since $(p-\sum\limits_{i=1}^{n}{{{x}_{i}}}){{p}^{r}}= \sum\limits_{i=1}^{n}{{{y}_{i}}}<n{{p}^{r}}$, we know there exists $1 \le k \le n-1$ such that
\begin{displaymath}
\left\{ \begin{array}{ll}
{{x}_{1}}+{{x}_{2}}+\cdots +{{x}_{n}}=p-k \\
{{y}_{1}}+{{y}_{2}}+\cdots +{{y}_{n}}=k{{p}^{r}}
\end{array} \right..
\end{displaymath}

Hence
\[{{S}_{n}^{(1)}}({{p}^{r+1}})=\sum\limits_{k=1}^{n-1}{\sum\limits_{\begin{smallmatrix}
 {{x}_{1}}+\cdots +{{x}_{n}}=p-k \\
 {{y}_{1}}+\cdots +{{y}_{n}}=kp^{r}
 \\
 x_{i} \ge 0, y_{i} < p^r, y_{i}\in {\mathcal{P}_{p}}
\end{smallmatrix}}{\frac{1}{({{x}_{1}}{{p}^{r}}+{{y}_{1}})\cdots ({{x}_{n}}{{p}^{r}}+{{y}_{n}})}}}.\]
Given $1 \le k \le n-1$, the equation ${{x}_{1}}+{{x}_{2}}+\cdots +{{x}_{n}}=p-k$ has $\binom {p-k+n-1}{n-1}$ solutions $(x_{1},x_{2}, \cdots , x_{n})$ of nonnegative integers.
Because
\[({{x}_{1}}{{p}^{r}}+{{y}_{1}})\cdots ({{x}_{n}}{{p}^{r}}+{{y}_{n}})\equiv \Big({{y}_{1}}\cdots {{y}_{n}}\sum\limits_{i=1}^{n}{\frac{{{x}_{i}}}{{{y}_{i}}}}\Big){{p}^{r}}+{{y}_{1}}\cdots {{y}_{n}}  \pmod{{{p}^{2r}}},\]
and
\[\sum\limits_{\begin{smallmatrix}{{x}_{1}}+\cdots +{{x}_{n}}=p-k\\ x_{i} \ge 0, 1 \le i \le n \end{smallmatrix}}{{{x}_{i}}}=\frac{1}{n}\sum\limits_{\begin{smallmatrix}{{x}_{1}}+\cdots +{{x}_{n}}=p-k \\ x_{i} \ge 0, 1 \le i \le n \end{smallmatrix}}{({{x}_{1}}+\cdots +{{x}_{n}})}=\frac{p-k}{n} \binom {p-k+n-1}{n-1} \equiv 0 \pmod{p}.\]
We have
\begin{displaymath}
\begin{split}
  {{S}_{n}^{(1)}}({{p}^{r+1}})  &\equiv \sum\limits_{k=1}^{n-1}{\sum\limits_{\begin{smallmatrix}{{x}_{1}}+\cdots +{{x}_{n}}=p-k \\ x_{i} \ge 0, 1 \le i \le n \end{smallmatrix}}{\sum\limits_{\begin{smallmatrix}
 {{y}_{1}}+\cdots +{{y}_{n}}=kp^{r} \\
  y_{i} < p^r, y_{i}\in {\mathcal{P}_{p}}, 1 \le i \le n
\end{smallmatrix}}{\frac{1}{({{y}_{1}}\cdots {{y}_{n}}\sum\limits_{i=1}^{n}{\frac{{{x}_{i}}}{{{y}_{i}}}}){{p}^{r}}+{{y}_{1}}\cdots {{y}_{n}}}}}} \\
 & \equiv \sum\limits_{k=1}^{n-1}{\sum\limits_{\begin{smallmatrix}{{x}_{1}}+\cdots +{{x}_{n}}=p-k \\ x_{i} \ge 0 , 1 \le i \le n\end{smallmatrix}}{\sum\limits_{\begin{smallmatrix}
 {{y}_{1}}+\cdots +{{y}_{n}}=kp^{r} \\
 y_{i} < p^r, y_{i}\in {\mathcal{P}_{p}}, 1 \le i \le n
\end{smallmatrix}}{\frac{{{y}_{1}}\cdots {{y}_{n}}-({{y}_{1}}\cdots {{y}_{n}}\sum\limits_{i=1}^{n}{\frac{{{x}_{i}}}{{{y}_{i}}}}){{p}^{r}}}{{{({{y}_{1}}\cdots {{y}_{n}})}^{2}}}}}}  \\
 & \equiv \sum\limits_{k=1}^{n-1}{\binom {p-k+n-1}{n-1} \sum\limits_{\begin{smallmatrix}
 {{y}_{1}}+\cdots +{{y}_{n}}=kp^{r} \\
 y_{i} < p^r, y_{i}\in {\mathcal{P}_{p}}, 1 \le i \le n
\end{smallmatrix}}{\frac{1}{{{y}_{1}}\cdots {{y}_{n}}}}} \\
& \quad -{{p}^{r}}\sum\limits_{k=1}^{n-1}{\sum\limits_{i=1}^{n}{\sum\limits_{\begin{smallmatrix}{{x}_{1}}+\cdots +{{x}_{n}}=p-k \\x_{j} \ge 0 , 1 \le j \le n\end{smallmatrix}}{{{x}_{i}}}\sum\limits_{\begin{smallmatrix}
 {{y}_{1}}+\cdots +{{y}_{n}}=kp^{r} \\
y_{j} < p^r, y_{j}\in {\mathcal{P}_{p}}, 1 \le j \le n
\end{smallmatrix}}{\frac{1}{{{y}_{1}}\cdots {{y}_{i-1}}y_{i}^{2}{{y}_{i+1}}\cdots {{y}_{n}}}}}} \\
 &  \equiv \sum\limits_{k=1}^{n-1}{\binom {p-k+n-1}{n-1} \sum\limits_{\begin{smallmatrix}
 {{y}_{1}}+\cdots +{{y}_{n}}=kp^{r} \\
 y_{i} < p^r, y_{i}\in {\mathcal{P}_{p}}, 1 \le i \le n
\end{smallmatrix}}{\frac{1}{{{y}_{1}}\cdots {{y}_{n}}}}}  \pmod{p^{r+1}}.
\end{split}
\end{displaymath}
This completes the proof of (i).
\end{proof}

\begin{lem}\label{Casolution}
Let $p>5$ be a prime. For $1 \le a \le 4$, we denote by ${{C}_{a}}$ the number of solutions $({{x}_{1}},\cdots ,{{x}_{5}})$ of nonnegative integers of the equation
\[{{x}_{1}}+\cdots +{{x}_{5}}=2p-a,\quad 0\le {{x}_{i}}<p, \quad 1 \le i \le 5.\]
Then we have\\
(i) ${{C}_{a}}= \binom {2p-a+4}{4} -5 \binom{p-a+4}{4}$. \\
(ii) ${{C}_{1}}\equiv -\frac{3}{4}p$ \text{\rm{(mod $p^2$)}}, ${{C}_{2}}\equiv \frac{1}{4}p$ \text{\rm{(mod $p^2$)}},  ${{C}_{3}}\equiv -\frac{1}{4}p$ \text{\rm{(mod ${p^2}$)}}, ${{C}_{4}}\equiv \frac{3}{4}p$ \text{\rm{(mod $p^2$)}}. \\
(iii) $\sum\limits_{\begin{smallmatrix}
 {{x}_{1}}+\cdots +{{x}_{5}}=2p-a \\
 0\le {{x}_{i}}<p, 1 \le i \le n
\end{smallmatrix}}{{{x}_{1}}}  \equiv 0$ \text{\rm{(mod $p$)}}.
\end{lem}
\begin{proof}
(i) Note that for any solution $({{x}_{1}},\cdots ,{{x}_{5}})$ of the equation
\[{{x}_{1}}+\cdots +{{x}_{5}}=2p-a, \quad {{x}_{i}}\ge 0, \quad 1 \le i \le 5,\]
at most one of ${{x}_{i}}\, (1 \le i \le 5)$ can be greater than or equal to $p$.

By Inclusion-Exclusion Principle, we deduce that
\begin{displaymath}
\begin{split}
{{C}_{a}} &=  \#\Big\{(x_{1},x_{2},x_{3},x_{4},x_{5})\left|{{{x}_{1}}+\cdots +{{x}_{5}}=2p-a, x_{i} \ge 0, 1 \le i \le 5}\right.\Big\} \\
 & \quad -5\sum\limits_{k=0}^{p-a}{\#\Big\{(x_{1},x_{2},x_{3},x_{4},x_{5})\left|{{{x}_{1}}+\cdots +{{x}_{4}}=p-a-k,{{x}_{5}}=p+k,0\le {{x}_{i}}<p, 1\le i\le 4}\right.\Big\}} \\
& = \binom {2p-a+4}{4} -5 \sum\limits_{k=0}^{p-a}{\binom {p-a-k+3}{3} } \\
& = \binom {2p-a+4}{4}- 5\sum\limits_{k=3}^{p-a+3}{\binom {k}{3}} \\
&=  \binom {2p-a+4}{4}- 5\binom{p-a+4}{4}. \\
\end{split}
\end{displaymath}

(ii) For $u=1$ or $2$ we have
\[ \binom {up-a+4}{4}=\frac {(up-a+4)(up-a+3)(up-a+2)(up-a+1)}{4!}.\]
Since $1 \le a \le 4$, we deduce that
\[\binom {up-a+4}{4} \equiv \frac{(-1)^{a-1}(a-1)!(4-a)!}{4!}up \pmod{p^2}.\]
From which the congruences in (ii) follows by (i) and some simple calculations.

(iii) By (ii) we have $C_{a} \equiv 0$ (mod $p$). Hence
\[\sum\limits_{\begin{smallmatrix}
 {{x}_{1}}+\cdots +{{x}_{5}}=2p-a \\
 0\le {{x}_{i}}<p, 1 \le i \le 5
\end{smallmatrix}}{{{x}_{1}}} =\frac{1}{5}\sum\limits_{\begin{smallmatrix}{{x}_{1}}+\cdots +{{x}_{5}}=2p-a \\ 0 \le x_{i} <p, 1 \le i \le 5 \end{smallmatrix}}{({{x}_{1}}+\cdots +{{x}_{5}})}=\frac{2p-a}{5}{{C}_{a}} \equiv 0 \pmod {p}.\]
\end{proof}

\begin{lem}\label{zhouxialem}(Cf. \cite{zhouxia})
Let $r,{{\alpha }_{1}},\cdots ,{{\alpha }_{n}}$ be positive integers, $r={{\alpha }_{1}}+\cdots +{{\alpha }_{n}}\le p-3$. Then
\[\sum\limits_{\begin{smallmatrix}
 1\le {{l}_{1}},\cdots ,{{l}_{n}}\le p-1 \\
 {{l}_{i}}\ne {{l}_{j}}, \forall i \ne j
\end{smallmatrix}}{\frac{1}{l_{1}^{{{\alpha }_{1}}}l_{2}^{{{\alpha }_{2}}}\cdots l_{n}^{{{\alpha }_{n}}}}}\equiv \left\{ \begin{array}{ll}
   {{(-1)}^{n}}(n-1)!\frac{r(r+1)}{2(r+2)}{{B}_{p-r-2}}{{p}^{2}}  \pmod{{{p}^{3}}},  & \hbox{if $2 \nmid r$}; \\
  {{(-1)}^{n-1}}(n-1)!\frac{r}{r+1}{{B}_{p-r-1}}p   \pmod{{{p}^{2}}}, & \hbox{if $2 | r$}. \\
\end{array} \right.\]
\end{lem}
This leads to the following corollary.
\begin{cor}\label{lemcor}
Let $\alpha$ be a positive integer and $p \ge \alpha +3$ a prime. Then
\[\sum\limits_{1 \le l <p}{\frac{1}{l^{\alpha}}} \equiv \left\{ \begin{array}{ll}
0 \pmod{p^2} & \hbox{if $2 \nmid \alpha$}; \\
0 \pmod{p} & \hbox{if $2 | \alpha$}.
\end{array} \right.\]
\end{cor}

\begin{lem}\label{2plem}
Let $p\ge 3$ be a prime, and ${{\alpha }_{1}},\cdots ,{{\alpha }_{n}}$ be positive integers, where $r={{\alpha }_{1}}+\cdots +{{\alpha }_{n}}\le p-3$.  We have
\begin{displaymath}
\sum\limits_{\begin{smallmatrix}
 1\le {{l}_{1}},\cdots ,{{l}_{n}}<2p \\
 {{l}_{i}}\ne {{l}_{j}},{{l}_{i}}\in {\mathcal{P}_{p}}
\end{smallmatrix}}{\frac{1}{l_{1}^{{{\alpha }_{1}}}l_{2}^{{{\alpha }_{2}}}\cdots l_{n}^{{{\alpha }_{n}}}}}\equiv \left\{ \begin{array}{ll}
 {{(-1)}^{n}}(n-1)!\frac{2r(r+1)}{r+2}{{B}_{p-r-2}}{{p}^{2}}  \pmod{{{p}^{3}}} & \hbox{if $2 \nmid r$}; \\
{{(-1)}^{n-1}}(n-1)!\frac{2r}{r+1}{{B}_{p-r-1}}p  \pmod{{{p}^{2}}} & \hbox{if $2 | r$}.
\end{array}\right.
\end{displaymath}
\end{lem}
\begin{proof}
For any positive integer $\alpha \le p-3$, we have
\begin{displaymath}
\begin{split}
   \sum\limits_{p<l<2p}{\frac{1}{{{l}^{\alpha }}}}&=\sum\limits_{0<l<p}{\frac{1}{{{(l+p)}^{\alpha }}}}\equiv \sum\limits_{0<l<p}{\frac{{{(l-p)}^{\alpha }}}{{{({{l}^{2}}-{{p}^{2}})}^{\alpha }}}} \\
 & \equiv \sum\limits_{0<l<p}{\frac{{{l}^{\alpha }}-\alpha p{{l}^{\alpha -1}}+\frac{\alpha (\alpha -1)}{2}{{p}^{2}}{{l}^{\alpha -2}}}{{{l}^{2\alpha }}-\alpha {{p}^{2}}{{l}^{2(\alpha -1)}}}} \\
 & \equiv \sum\limits_{0<l<p}{\frac{({{l}^{2\alpha }}+\alpha {{p}^{2}}{{l}^{2(\alpha -1)}})\big({{l}^{\alpha }}-\alpha p{{l}^{\alpha -1}}+\frac{\alpha (\alpha -1)}{2}{{p}^{2}}{{l}^{\alpha -2}}\big)}{{{l}^{4\alpha }}}} \\
 & \equiv \sum\limits_{0<l<p}{\frac{1}{{{l}^{\alpha }}}}-\alpha p\sum\limits_{0<l<p}{\frac{1}{{{l}^{\alpha +1}}}}+\frac{\alpha (\alpha +1)}{2}{{p}^{2}}\sum\limits_{0<l<p}{\frac{1}{{{l}^{\alpha +2}}}}  \pmod{{{p}^{3}}}. \\
\end{split}
\end{displaymath}
Because $\sum\nolimits_{0<l<p}{\frac{1}{{{l}^{\alpha +2}}}}\equiv 0$ (mod $p$), we have
\[\sum\limits_{\begin{smallmatrix}
 0<l<2p \\
 l\ne p
\end{smallmatrix}}{\frac{1}{{{l}^{\alpha }}}}=\sum\limits_{1\le l<p}{\frac{1}{{{l}^{\alpha }}}}+\sum\limits_{p<l<2p}{\frac{1}{{{l}^{\alpha }}}}\equiv 2\sum\limits_{0<l<p}{\frac{1}{{{l}^{\alpha }}}}-\alpha p\sum\limits_{0<l<p}{\frac{1}{{{l}^{\alpha +1}}}} \pmod{{{p}^{3}}}.\]
By Lemma \ref{zhouxialem}, we obtain
\begin{displaymath}
\sum\limits_{\begin{smallmatrix}0<l<2p\\ l \ne p \end{smallmatrix}}{\frac{1}{{{l}^{\alpha }}}}\equiv  \left\{ \begin{array}{ll}
-\frac{2\alpha (\alpha +1)}{\alpha +2}{{p}^{2}}{{B}_{p-\alpha -2}}  \pmod{{{p}^{3}}} & \hbox{if $2 \nmid \alpha $};\\
\frac{2\alpha }{\alpha +1}p{{B}_{p-\alpha -1}} \pmod{{{p}^{2}}} &  \hbox{if $ 2| \alpha$}.
\end{array} \right.
\end{displaymath}
This proves the lemma for $n=1$. Now assume the lemma is true when the number of variables is less than $n$.
We have
\begin{displaymath}
\begin{split}
   \sum\limits_{\begin{smallmatrix}
 1\le {{l}_{1}},\cdots ,{{l}_{n}}<2p \\
 {{l}_{i}}\ne {{l}_{j}}, l_{i} \in \mathcal{P}_{p}
\end{smallmatrix}}{\frac{1}{l_{1}^{{{\alpha }_{1}}}\cdots l_{n}^{{{\alpha }_{n}}}}}&=\sum\limits_{\begin{smallmatrix}
 1\le {{l}_{1}},\cdots ,{{l}_{n-1}}<2p \\
 {{l}_{i}}\ne {{l}_{j}}, l_{i} \in \mathcal{P}_{p}
\end{smallmatrix}}{\frac{1}{l_{1}^{{{\alpha }_{1}}}\cdots l_{n-1}^{{{\alpha }_{n-1}}}}}\Big(\sum\limits_{\begin{smallmatrix}
 1\le {{l}_{n}}<2p \\
 {{l}_{n}}\ne p
\end{smallmatrix}}{\frac{1}{l_{n}^{{{\alpha }_{n}}}}}-\sum\limits_{i=1}^{n-1}{\frac{1}{l_{i}^{{\alpha_{n}}}}}\Big) \\
 & = \Big(\sum\limits_{\begin{smallmatrix}
 1\le {{l}_{1}},\cdots ,{{l}_{n-1}}<2p \\
 {{l}_{i}}\ne {{l}_{j}}, l_{i} \in \mathcal{P}_{p}
\end{smallmatrix}}{\frac{1}{l_{1}^{{{\alpha }_{1}}}\cdots l_{n-1}^{{{\alpha }_{n-1}}}}}\Big)\Big(\sum\limits_{\begin{smallmatrix}
 1\le {{l}_{n}}<2p \\
 {{l}_{n}}\ne p
\end{smallmatrix}}{\frac{1}{l_{n}^{{{\alpha }_{n}}}}}\Big) \\
 & \quad -\sum\limits_{\begin{smallmatrix}
 1\le {{l}_{1}},\cdots ,{{l}_{n-1}}<2p \\
 {{l}_{i}}\ne {{l}_{j}}, l_{i} \in \mathcal{P}_{p}
\end{smallmatrix}}{\frac{1}{l_{1}^{{{\alpha }_{1}}+{{\alpha }_{n}}}\cdots l_{n-1}^{{{\alpha }_{n-1}}}}}-\cdots  \\
 & \quad -\sum\limits_{\begin{smallmatrix}
 1\le {{l}_{1}},\cdots ,{{l}_{n-1}}<2p \\
 {{l}_{i}}\ne {{l}_{j}}, l_{i} \in \mathcal{P}_{p}
\end{smallmatrix}}{\frac{1}{l_{1}^{{{\alpha }_{1}}}\cdots l_{n-1}^{{{\alpha }_{n-1}}+{{\alpha }_{n}}}}} .\\
\end{split}
\end{displaymath}
From the assumption, we have
\begin{displaymath}
\Big(\sum\limits_{\begin{smallmatrix}
 1\le {{l}_{1}},\cdots ,{{l}_{n-1}}<2p \\
 {{l}_{i}}\ne {{l}_{j}}, l_{i} \in \mathcal{P}_{p}
\end{smallmatrix}}{\frac{1}{l_{1}^{{{\alpha }_{1}}}\cdots l_{n-1}^{{{\alpha }_{n-1}}}}}\Big)\Big(\sum\limits_{\begin{smallmatrix}
 1\le {{l}_{n}}<2p \\
 {{l}_{n}}\ne p
\end{smallmatrix}}{\frac{1}{l_{n}^{{{\alpha }_{n}}}}}\Big)\equiv   \left\{ \begin{array}{ll}
0  \pmod{p^3} & \hbox{if $ 2 \nmid r$}; \\
0  \pmod{p^2} & \hbox{if $ 2| r$}.
\end{array} \right.
\end{displaymath}

If $r$ is odd, then
\begin{displaymath}
\begin{split}
\sum\limits_{\begin{smallmatrix}
 1\le {{l}_{1}},\cdots ,{{l}_{n}}<2p \\
 {{l}_{i}}\ne {{l}_{j}},{{l}_{i}}\in {\mathcal{P}_{p}}
\end{smallmatrix}}{\frac{1}{l_{1}^{{{\alpha }_{1}}}l_{2}^{{{\alpha }_{2}}}\cdots l_{n}^{{{\alpha }_{n}}}}}   &  \equiv -(n-1)\sum\limits_{\begin{smallmatrix}
 1\le {{l}_{1}},\cdots ,{{l}_{n-1}}<2p \\
 {{l}_{i}}\ne {{l}_{j}},{{l}_{i}}\in {\mathcal{P}_{p}}
\end{smallmatrix}}{\frac{1}{l_{1}^{{{\alpha }_{1}}+{\alpha}_{n}}l_{2}^{{{\alpha }_{2}}}\cdots l_{n-1}^{{{\alpha }_{n-1}}}}} \\
 & \equiv -(n-1){{(-1)}^{n-1}}(n-2)!\frac{2r(r+1)}{r+2}{{B}_{p-r-2}}{{p}^{2}} \\
 & \equiv {{(-1)}^{n}}(n-1)!\frac{2r(r+1)}{r+2}{{B}_{p-r-2}}{{p}^{2}} \pmod{{{p}^{3}}}.
\end{split}
\end{displaymath}

If $r$ is even, similarly we can derive
\[\sum\limits_{\begin{smallmatrix}
 1\le {{l}_{1}},\cdots ,{{l}_{n}}<2p \\
 {{l}_{i}}\ne {{l}_{j}},{{l}_{i}}\in {\mathcal{P}_{p}}
\end{smallmatrix}}{\frac{1}{l_{1}^{{{\alpha }_{1}}}l_{2}^{{{\alpha }_{2}}}\cdots l_{n}^{{{\alpha }_{n}}}}}\equiv {{(-1)}^{n-1}}(n-1)!\frac{2r}{r+1}{{B}_{p-r-1}}p  \pmod{{{p}^{2}}}.\]
The proof of Lemma \ref{2plem} is complete by induction on $n$.
\end{proof}
By letting $r=n=1$ in this lemma, we obtain the following corollary.
\begin{cor}\label{lemcor2}
Let $\alpha$ be a positive integer and $p \ge \alpha +3$ be a prime. Then
\begin{displaymath}
\sum\limits_{1 \le l <2p, l \ne p}{\frac{1}{l^{\alpha}}} \equiv \left\{ \begin{array}{ll}
0 \pmod{p^{2}} & \hbox{if $2 \nmid \alpha$}; \\
0 \pmod{p} & \hbox{if $2 | \alpha$}.
\end{array}\right.
\end{displaymath}
\end{cor}
\begin{lem}\label{S52modp}
Let $p>5$ be a prime. We have
\[\sum\limits_{\begin{smallmatrix}
 {{l}_{1}}+\cdots +{{l}_{5}}=2p \\
 1\le {{l}_{1}},\cdots ,{{l}_{5}}<p
\end{smallmatrix}}{\frac{1}{{{l}_{1}}{{l}_{2}}{{l}_{3}}{{l}_{4}}{{l}_{5}}}}\equiv 2\cdot 4!{{B}_{p-5}} \pmod{p}.\]
\end{lem}

\begin{proof}
Let ${{u}_{i}}={{l}_{1}}+\cdots +{{l}_{i}},1\le i\le 4.$ We have
\begin{equation}\label{start}
\begin{split}
   \sum\limits_{\begin{smallmatrix}
 {{l}_{1}}+\cdots +{{l}_{5}} =2p \\
 {{l}_{_{1}}},\cdots ,{{l}_{5}}\in {\mathcal{P}_{p}}
\end{smallmatrix}}{\frac{1}{{{l}_{1}}{{l}_{2}}{{l}_{3}}{{l}_{4}}{{l}_{5}}}} & =\frac{1}{2p}\sum\limits_{\begin{smallmatrix}
 {{l}_{1}}+\cdots +{{l}_{5}}=2p \\
 {{l}_{1}},\cdots ,{{l}_{5}}\in {\mathcal{P}_{p}}
\end{smallmatrix}}{\frac{{{l}_{1}}+{{l}_{2}}+\cdots +{{l}_{5}}}{{{l}_{1}}{{l}_{2}}\cdots {{l}_{5}}}} \\
& =\frac{5}{2p}\sum\limits_{\begin{smallmatrix}
 {{l}_{1}}+\cdots +{{l}_{4}}<2p \\
 {{l}_{1}},\cdots ,{{l}_{4}},{{u}_{4}}\in {\mathcal{P}_{p}}
\end{smallmatrix}}{\frac{1}{{{l}_{1}}{{l}_{2}}{{l}_{3}}{{l}_{4}}}} \\
 & =\frac{5\cdot 4}{2p}\sum\limits_{\begin{smallmatrix}
 {{l}_{1}},{{l}_{2}},{{l}_{3}}<{{u}_{4}}<2p \\
 {{l}_{1}},{l}_{2} ,{{l}_{3}},{{u}_{4}},{{u}_{4}}-{{u}_{3}}\in {\mathcal{P}_{p}}
\end{smallmatrix}}{\frac{1}{{{l}_{1}}{{l}_{2}}{{l}_{3}}u_{4}}}=\cdots  \\
 & =\frac{5!}{2p}\sum\limits_{\begin{smallmatrix}
 1\le {{u}_{1}}<\cdots <{{u}_{4}}<2p \\
 {{u}_{1}},{{u}_{2}}-{{u}_{1}},\cdots ,{{u}_{4}}{{-}{u}_{3}},{{u}_{4}}\in {\mathcal{P}_{p}}
\end{smallmatrix}}{\frac{1}{{{u}_{1}}{{u}_{2}}{{u}_{3}}{{u}_{4}}}}. \\
\end{split}
\end{equation}
Moreover,
\begin{equation}\label{reason}
\begin{split}
   \sum\limits_{\begin{smallmatrix}
 1\le {{u}_{1}}<\cdots <{{u}_{4}}<2p \\
 {{u}_{1}},{{u}_{2}}-{{u}_{1}},{u}_{3}-{u}_{2} ,{{u}_{4}}-{{u}_{3}},{{u}_{4}}\in {\mathcal{P}_{p}}
\end{smallmatrix}}{\frac{1}{{{u}_{1}}{{u}_{2}}{{u}_{3}}{{u}_{4}}}}
 & =\sum\limits_{\begin{smallmatrix}
 1\le {{u}_{1}}<\cdots <{{u}_{4}}<2p \\
 {{u}_{1}},{{u}_{2}},{u}_{3} ,{{u}_{4}}\in {\mathcal{P}_{p}}
 \\
 {{u}_{2}}-{{u}_{1}},{u}_{3}-{u}_{2} ,{{u}_{4}}-{{u}_{3}}\in {\mathcal{P}_{p}}
\end{smallmatrix}}{\frac{1}{{{u}_{1}}{{u}_{2}}{{u}_{3}}{{u}_{4}}}}   \\
 & \quad +\sum\limits_{\begin{smallmatrix}
 1\le {{u}_{1}}<\cdots <{{u}_{4}}<2p,{{u}_{2}}=p \\
 {{u}_{1}},{{u}_{2}}-{{u}_{1}},{u}_{3}-{u}_{2},{{u}_{4}}-{{u}_{3}},{{u}_{4}}\in {\mathcal{P}_{p}}
\end{smallmatrix}}{\frac{1}{{{u}_{1}}{{u}_{2}}{{u}_{3}}{{u}_{4}}}} \\
&\quad +\sum\limits_{\begin{smallmatrix}
 1\le {{u}_{1}}<\cdots <{{u}_{4}}<2p,{{u}_{3}}=p \\
 {{u}_{1}},{{u}_{2}}-{{u}_{1}},{u}_{3}-{u}_{2},{{u}_{4}}-{{u}_{3}},{{u}_{4}}\in {\mathcal{P}_{p}}
\end{smallmatrix}}{\frac{1}{{{u}_{1}}{{u}_{2}}{{u}_{3}}{{u}_{4}}}}. \\
\end{split}
\end{equation}
It is not hard to see
\begin{displaymath}
\begin{split}
  \sum\limits_{\begin{smallmatrix}
 1\le {{u}_{1}}<\cdots <{{u}_{4}}<2p,{{u}_{2}}=p \\
 {{u}_{1}},{{u}_{2}}-{{u}_{1}},{u}_{3}-{u}_{2} ,{{u}_{4}}-{{u}_{3}},{{u}_{4}}\in {\mathcal{P}_{p}}
\end{smallmatrix}}{\frac{1}{{{u}_{1}}{{u}_{2}}{{u}_{3}}{{u}_{4}}}} &=\frac{1}{p}\sum\limits_{1\le {{u}_{1}}<p}{\frac{1}{{{u}_{1}}}}\sum\limits_{\begin{smallmatrix}
 p<{{u}_{3}}<{{u}_{4}}<2p \\
 {{u}_{3}},{{u}_{4}}-{{u}_{3}},{{u}_{4}}\in {\mathcal{P}_{p}}
\end{smallmatrix}}{\frac{1}{{{u}_{3}}{{u}_{4}}}} \\
 & =\frac{1}{2p}\Big(\sum\limits_{1\le {{u}_{1}}<p}{\frac{1}{{{u}_{1}}}}\Big)\bigg({{\Big(\sum\limits_{p < {{u}_{3}}<2p}{\frac{1}{{{u}_{3}}}}\Big)}^{2}}-\sum\limits_{p < {{u}_{3}}<2p}{\frac{1}{{{u}_{3}^{2}}}}\bigg)\\
 & \equiv 0  \pmod{{{p}^{2}}} .
\end{split}
\end{displaymath}
Here in the last congruence we use Corollary  \ref{lemcor} and Corollary \ref{lemcor2}.

In the same way, we have
\begin{displaymath}
\begin{split}
\sum\limits_{\begin{smallmatrix}
 1\le {{u}_{1}}<\cdots <{{u}_{4}}<2p,{{u}_{3}}=p \\
 {{u}_{1}},{{u}_{2}}-{{u}_{1}},{u}_{3}-{u}_{2},{{u}_{4}}-{{u}_{3}},{{u}_{4}}\in {\mathcal{P}_{p}}
\end{smallmatrix}}{\frac{1}{{{u}_{1}}{{u}_{2}}{{u}_{3}}{{u}_{4}}}}& =\frac{1}{p}\sum\limits_{\begin{smallmatrix}
 1\le {{u}_{1}}<{{u}_{2}}<p \\
 {{u}_{1}},{{u}_{2}}-{{u}_{1}},{{u}_{2}}\in {\mathcal{P}_{p}}
\end{smallmatrix}}{\frac{1}{{{u}_{1}}{{u}_{2}}}}\sum\limits_{p<{{u}_{4}}<2p}{\frac{1}{{{u}_{4}}}} \\
& =\frac{1}{p}\bigg({{\Big(\sum\limits_{1\le {{u}_{1}}<p}{\frac{1}{{{u}_{1}}}}\Big)}^{2}}-\sum\limits_{1\le {{u}_{1}}<p}{\frac{1}{u_{1}^{2}}}\Bigg)\Big(\sum\limits_{p<{{u}_{4}}<2p}{\frac{1}{{{u}_{4}}}}\Big) \\
&\equiv 0 \pmod{{{p}^{2}}}.
\end{split}
\end{displaymath}
Hence from (\ref{reason}) we deduce that
\begin{equation}\label{midterm}
\begin{split}
\sum\limits_{\begin{smallmatrix}
 1\le {{u}_{1}}<\cdots <{{u}_{4}}<2p \\
 {{u}_{1}},{{u}_{2}}-{{u}_{1}},{u}_{3}-{u}_{2},{{u}_{4}}-{{u}_{3}},{{u}_{4}}\in {\mathcal{P}_{p}}
\end{smallmatrix}}{\frac{1}{{{u}_{1}}{{u}_{2}}{{u}_{3}}{{u}_{4}}}} &\equiv \sum\limits_{\begin{smallmatrix}
 1\le {{u}_{1}}<\cdots <{{u}_{4}}<2p \\
 {{u}_{1}},{{u}_{2}},{u}_{3} ,{{u}_{4}}\in {\mathcal{P}_{p}}
 \\
 {{u}_{2}}-{{u}_{1}},{u}_{3}-{u}_{2} ,{{u}_{4}}-{{u}_{3}}\in {\mathcal{P}_{p}}
\end{smallmatrix}}{\frac{1}{{{u}_{1}}{{u}_{2}}{{u}_{3}}{{u}_{4}}}} \\
& \equiv \sum\limits_{\begin{smallmatrix}
 1\le {{u}_{1}}<\cdots <{{u}_{4}}<2p \\
 {{u}_{1}},{{u}_{2}},{u}_{3},{{u}_{4}}\in {\mathcal{P}_{p}}
\end{smallmatrix}}{\frac{1}{{{u}_{1}}{{u}_{2}}{{u}_{3}}{{u}_{4}}}}-{{T}_{1}}-{{T}_{2}}-{{T}_{3}} \pmod{{{p}^{2}}},
\end{split}
\end{equation}
where
\begin{displaymath}
\begin{split}
{{T}_{1}}&=\sum\limits_{\begin{smallmatrix}
 {{u}_{1}}<{{u}_{1}}+p<{{u}_{3}}<{{u}_{4}}<2p \\
 {{u}_{1}},{{u}_{3}},{{u}_{4}}\in {\mathcal{P}_{p}}
\end{smallmatrix}}{\frac{1}{{{u}_{1}}({{u}_{1}}+p){{u}_{3}}{{u}_{4}}}}, \quad {{T}_{2}}=\sum\limits_{\begin{smallmatrix}
 {{u}_{1}}<{{u}_{2}}<{{u}_{2}}+p<{{u}_{4}}<2p \\
 {{u}_{1}},{{u}_{2}},{{u}_{4}}\in {\mathcal{P}_{p}}
\end{smallmatrix}}{\frac{1}{{{u}_{1}}{{u}_{2}}({{u}_{2}}+p){{u}_{4}}}}, \\
& {{T}_{3}}=\sum\limits_{\begin{smallmatrix}
 {{u}_{1}}<{{u}_{2}}<{{u}_{3}}<{{u}_{3}}+p<2p \\
 {{u}_{1}},{{u}_{2}},{{u}_{3}}\in {\mathcal{P}_{p}}
\end{smallmatrix}}{\frac{1}{{{u}_{1}}{u}_{2}{{u}_{3}}{({u}_{3}+p)}}}.
\end{split}
\end{displaymath}

We have
\begin{displaymath}
\begin{split}
   {{T}_{1}} &=\sum\limits_{\begin{smallmatrix}
 {{u}_{1}}<{{u}_{1}}+p<{{u}_{3}}<{{u}_{4}}<2p \\
 {{u}_{1}},{{u}_{3}},{{u}_{4}}\in {\mathcal{P}_{p}}
\end{smallmatrix}}{\frac{1}{{{u}_{1}}({{u}_{1}}+p){{u}_{3}}{{u}_{4}}}}  \quad \hbox{(replace ${u}_{3}$ by ${u}_{3}+p$ and ${u}_{4}$ by ${u}_{4}+p$)} \\
 & =\sum\limits_{\begin{smallmatrix}
 {{u}_{1}}<{{u}_{3}}<{{u}_{4}}<p \\
 {{u}_{1}},{{u}_{3}},{{u}_{4}}\in {\mathcal{P}_{p}}
\end{smallmatrix}}{\frac{1}{{{u}_{1}}({{u}_{1}}+p)({{u}_{3}}+p)({{u}_{4}}+p)}} \\
 & \equiv \sum\limits_{\begin{smallmatrix}
 {{u}_{1}}<{{u}_{3}}<{{u}_{4}}<p \\
 {{u}_{1}},{{u}_{3}},{{u}_{4}}\in {\mathcal{P}_{p}}
\end{smallmatrix}}{\frac{{{u}_{1}}{{u}_{3}}{{u}_{4}}-p({{u}_{1}}{{u}_{3}}+{{u}_{3}}{{u}_{4}}+{{u}_{4}}{{u}_{1}})}{{{u}_{1}}u_{1}^{2}u_{3}^{2}u_{4}^{2}}} \\
 & \equiv \sum\limits_{\begin{smallmatrix}
 {{u}_{1}}<{{u}_{3}}<{{u}_{4}}<p \\
 {{u}_{1}},{{u}_{3}},{{u}_{4}}\in {\mathcal{P}_{p}}
\end{smallmatrix}}{\frac{1}{u_{1}^{2}{{u}_{3}}{{u}_{4}}}}-p\sum\limits_{\begin{smallmatrix}
 {{u}_{1}}<{{u}_{3}}<{{u}_{4}}<p \\
 {{u}_{1}},{{u}_{3}},{{u}_{4}}\in {\mathcal{P}_{p}}
\end{smallmatrix}}{\frac{1}{u_{1}^{2}{{u}_{3}}u_{4}^{2}}}-p\sum\limits_{\begin{smallmatrix}
 {{u}_{1}}<{{u}_{3}}<{{u}_{4}}<p \\
 {{u}_{1}},{{u}_{3}},{{u}_{4}}\in {\mathcal{P}_{p}}
\end{smallmatrix}}{\frac{1}{u_{1}^{3}{{u}_{3}}{{u}_{4}}}} \\
& \quad -p\sum\limits_{\begin{smallmatrix}
 {{u}_{1}}<{{u}_{3}}<{{u}_{4}}<p \\
 {{u}_{1}},{{u}_{3}},{{u}_{4}}\in {\mathcal{P}_{p}}
\end{smallmatrix}}{\frac{1}{u_{1}^{2}u_{3}^{2}{{u}_{4}}}}\\
& \equiv \sum\limits_{\begin{smallmatrix}
 {{u}_{1}}<{{u}_{2}}<{{u}_{3}}<p \\
 {{u}_{1}},{{u}_{2}},{{u}_{3}}\in {\mathcal{P}_{p}}
\end{smallmatrix}}{\frac{1}{u_{1}^{2}{{u}_{2}}{{u}_{3}}}}-p\sum\limits_{\begin{smallmatrix}
 {{u}_{1}}<{{u}_{2}}<{{u}_{3}}<p \\
 {{u}_{1}},{{u}_{2}},{{u}_{3}}\in {\mathcal{P}_{p}}
\end{smallmatrix}}{\frac{1}{u_{1}^{2}{{u}_{2}}u_{3}^{2}}}-p\sum\limits_{\begin{smallmatrix}
 {{u}_{1}}<{{u}_{2}}<{{u}_{3}}<p \\
 {{u}_{1}},{{u}_{2}},{{u}_{3}}\in {\mathcal{P}_{p}}
\end{smallmatrix}}{\frac{1}{u_{1}^{3}{{u}_{2}}{{u}_{3}}}} \\
& \quad -p\sum\limits_{\begin{smallmatrix}
 {{u}_{1}}<{{u}_{2}}<{{u}_{3}}<p \\
 {{u}_{1}},{{u}_{2}},{{u}_{3}}\in {\mathcal{P}_{p}}
\end{smallmatrix}}{\frac{1}{u_{1}^{2}u_{2}^{2}{{u}_{3}}}} \pmod{{{p}^{2}}}.
\end{split}
\end{displaymath}
Here in the last congruence we just replace the variables $u_{3}$ by $u_{2}$ and $u_{4}$ by $u_{3}$.

Similarly,
\begin{displaymath}
\begin{split}
   {{T}_{2}}&=\sum\limits_{{{u}_{1}}<{{u}_{2}}<{{u}_{2}}+p<{{u}_{4}}<2p}{\frac{1}{{{u}_{1}}{{u}_{2}}({{u}_{2}}+p){{u}_{4}}}} \quad \hbox{(replace $u_{4}$ by $u_{3}+p$)} \\
 & =\sum\limits_{{{u}_{1}}<{{u}_{2}}<{{u}_{3}}<p}{\frac{1}{{{u}_{1}}{{u}_{2}}({{u}_{2}}+p)({{u}_{3}}+p)}} \\
 & \equiv \sum\limits_{{{u}_{1}}<{{u}_{2}}<{{u}_{3}}<p}{\frac{{{u}_{2}}{{u}_{3}}-p({{u}_{2}}+{{u}_{3}})}{{{u}_{1}}{{u}_{2}}u_{2}^{2}u_{3}^{2}}} \\
 & \equiv \sum\limits_{{{u}_{1}}<{{u}_{2}}<{{u}_{3}}<p}{\frac{1}{{{u}_{1}}u_{2}^{2}{{u}_{3}}}}-p\sum\limits_{{{u}_{1}}<{{u}_{2}}<{{u}_{3}}<p}{\frac{1}{{{u}_{1}}u_{2}^{2}u_{3}^{2}}}-p\sum\limits_{{{u}_{1}}<{{u}_{2}}<{{u}_{3}}<p}{\frac{1}{{{u}_{1}}u_{2}^{3}{{u}_{3}}}} \pmod{p^2}
\end{split}
\end{displaymath}
and
\begin{displaymath}
\begin{split}
  {{T}_{3}} &=\sum\limits_{{{u}_{1}}<{{u}_{2}}<{{u}_{3}}<p}{\frac{1}{{{u}_{1}}{{u}_{2}}{{u}_{3}}({{u}_{3}}+p)}} \\
 & \equiv \sum\limits_{{{u}_{1}}<{{u}_{2}}<{{u}_{3}}<p}{\frac{{{u}_{3}}-p}{{{u}_{1}}{{u}_{2}}u_{3}^{3}}} \\
 & \equiv \sum\limits_{{{u}_{1}}<{{u}_{2}}<{{u}_{3}}<p}{\frac{1}{{{u}_{1}}{{u}_{2}}u_{3}^{2}}}-p\sum\limits_{{{u}_{1}}<{{u}_{2}}<{{u}_{3}}<p}{\frac{1}{{{u}_{1}}{{u}_{2}}u_{3}^{3}}} \pmod{p^2}. \\
\end{split}
\end{displaymath}

Hence by Lemma \ref{zhouxialem}, we have
\begin{displaymath}
\begin{split}
{{T}_{1}}+{{T}_{2}}+{{T}_{3}}& =\sum\limits_{{{u}_{1}}<{{u}_{2}}<{{u}_{3}}<p}{\Big(\frac{1}{u_{1}^{2}{{u}_{2}}{{u}_{3}}}+\frac{1}{{{u}_{1}}u_{2}^{2}{{u}_{3}}}+\frac{1}{{{u}_{1}}{{u}_{2}}u_{3}^{3}}\Big)} \\
 & \quad -p\sum\limits_{{{u}_{1}}<{{u}_{2}}<{{u}_{3}}<p}{\Big(\frac{1}{u_{1}^{2}{{u}_{2}}u_{3}^{2}}+\frac{1}{u_{1}^{2}u_{2}^{2}{{u}_{3}}}+\frac{1}{{{u}_{1}}u_{2}^{2}u_{3}^{2}}\Big)} \\
 & \quad -p\sum\limits_{{{u}_{1}}<{{u}_{2}}<{{u}_{3}}<p}{\Big(\frac{1}{u_{1}^{3}{{u}_{2}}{{u}_{3}}}+\frac{1}{{{u}_{1}}u_{2}^{3}{{u}_{3}}}+\frac{1}{{{u}_{1}}{{u}_{2}}u_{3}^{3}}\Big)} \\
&\equiv \frac{1}{2}\Bigg(\sum\limits_{\begin{smallmatrix}
 {{u}_{1}},{{u}_{2}},{{u}_{3}}<p \\
 {{u}_{i}}\ne {{u}_{j}}
\end{smallmatrix}}{\frac{1}{u_{1}^{2}{{u}_{2}}{{u}_{3}}}}-p\sum\limits_{\begin{smallmatrix}
 {{u}_{1}},{{u}_{2}},{{u}_{3}}<p \\
 {{u}_{i}}\ne {{u}_{j}}
\end{smallmatrix}}{\frac{1}{u_{1}^{2}{{u}_{2}}u_{3}^{2}}}-p\sum\limits_{\begin{smallmatrix}
 {{u}_{1}},{{u}_{2}},{{u}_{3}}<p \\
 {{u}_{i}}\ne {{u}_{j}}
\end{smallmatrix}}{\frac{1}{u_{1}^{3}{{u}_{2}}{{u}_{3}}}}\Bigg) \\
& \equiv \frac{4}{5}p{{B}_{p-5}} \pmod{{{p}^{2}}}.
\end{split}
\end{displaymath}
Substituting this congruence into (\ref{midterm}) and combining with (\ref{start}),  we obtain
\begin{displaymath}
\sum\limits_{\begin{smallmatrix}
 {{l}_{1}}+\cdots +{{l}_{5}}=2p \\
 {{l}_{_{1}}},\cdots ,{{l}_{5}}\in {\mathcal{P}_{p}}
\end{smallmatrix}}{\frac{1}{{{l}_{1}}{{l}_{2}}{{l}_{3}}{{l}_{4}}{{l}_{5}}}}\equiv \frac{5!}{2p}\Big(\sum\limits_{\begin{smallmatrix}
 1\le {{u}_{1}}<\cdots <{{u}_{4}}<2p \\
 {{u}_{1}},{{u}_{2}},{u}_{3} ,{{u}_{4}}\in {\mathcal{P}_{p}}
\end{smallmatrix}}{\frac{1}{{{u}_{1}}{{u}_{2}}{{u}_{3}}{{u}_{4}}}}-\frac{4}{5}p{{B}_{p-5}}\Big) \pmod{p}.
\end{displaymath}
By Lemma \ref{2plem}, we have
\[\sum\limits_{\begin{smallmatrix}
 1\le {{u}_{1}}<\cdots <{{u}_{4}}<2p \\
 {{u}_{1}},{{u}_{2}},{u}_{3} ,{{u}_{4}}\in {\mathcal{P}_{p}}
\end{smallmatrix}}{\frac{1}{{{u}_{1}}{{u}_{2}}{{u}_{3}}{{u}_{4}}}} = \frac{1}{4!} \sum\limits_{\begin{smallmatrix}
 1\le {{u}_{1}},\cdots ,{{u}_{4}}<2p \\
 {{u}_{i}}\ne {{u}_{j}},{{u}_{i}}\in {\mathcal{P}_{p}}
\end{smallmatrix}}{\frac{1}{{{u}_{1}}{{u}_{2}}{{u}_{3}}{{u}_{4}}}}\equiv -\frac{2}{5}p{{B}_{p-5}} \pmod{{{p}^{2}}}.\]
Hence
\[\sum\limits_{\begin{smallmatrix}
 {{l}_{1}}+\cdots +{{l}_{5}}=2p \\
 {{l}_{_{1}}},\cdots ,{{l}_{5}}\in {\mathcal{P}_{p}}
\end{smallmatrix}}{\frac{1}{{{l}_{1}}{{l}_{2}}{{l}_{3}}{{l}_{4}}{{l}_{5}}}}\equiv  -3\cdot 4!{{B}_{p-5}} \pmod{p}.\]

We observe that
\begin{displaymath}
\begin{split}
\sum\limits_{\begin{smallmatrix}
 {{l}_{1}}+{{l}_{2}}+\cdots +{{l}_{5}}=2p \\
 {{l}_{1}},\cdots ,{{l}_{5}}\in {\mathcal{P}_{p}}
\end{smallmatrix}}{\frac{1}{{{l}_{1}}{{l}_{2}}{{l}_{3}}{{l}_{4}}{{l}_{5}}}}&=\sum\limits_{\begin{smallmatrix}
 {{l}_{1}}+{{l}_{2}}+\cdots +{{l}_{5}}=2p \\
 {{l}_{1}},\cdots ,{{l}_{5}}<p
\end{smallmatrix}}{\frac{1}{{{l}_{1}}{{l}_{2}}{{l}_{3}}{{l}_{4}}{{l}_{5}}}}+5\sum\limits_{\begin{smallmatrix}
 {{l}_{1}}+{{l}_{2}}+\cdots +{{l}_{5}}=p \\
 {{l}_{1}}, \cdots ,{{l}_{5}}<p
\end{smallmatrix}}{\frac{1}{({{l}_{1}}+p){{l}_{2}}{{l}_{3}}{{l}_{4}}{{l}_{5}}}} \\
& \equiv S_{5}^{(2)}(p)+5S_{5}^{(1)}(p) \pmod{p}.
\end{split}
\end{displaymath}
 By (\ref{zhou}), we have $S_{5}^{(1)}(p)\equiv -4!{{B}_{p-5}}$ (mod $p$). Hence we deduce that
\[S_{5}^{(2)}(p)\equiv \sum\limits_{\begin{smallmatrix}
 {{l}_{1}}+{{l}_{2}}+\cdots +{{l}_{5}}=2p \\
 {{l}_{1}},\cdots ,{{l}_{5}}\in {\mathcal{P}_{p}}
\end{smallmatrix}}{\frac{1}{{{l}_{1}}{{l}_{2}}{{l}_{3}}{{l}_{4}}{{l}_{5}}}}-5S_{5}^{(1)}(p)\equiv 2\cdot 4!{{B}_{p-5}} \pmod {p}.\]
This completes the proof.
\end{proof}

% -----------------------------------------------------------

\section{Proofs of the Theorems}

\begin{proof}[Proof of Theorem \ref{main}]
For any 5-tuples $({{l}_{1}},\cdots ,{{l}_{5}})$ of integers satisfying ${{l}_{1}}+\cdots +{{l}_{5}}=2{{p}^{r+1}}$, $1 \le l_{i} <p^{r+1}$, $l_{i} \in \mathcal{P}_{p}$, $1 \le i \le 5$, we rewrite them as
\[{{l}_{i}}={{x}_{i}}{{p}^{r}}+{{y}_{i}},\quad 0\le {{x}_{i}}<p,\quad 1\le {{y}_{i}}<{{p}^{r}}, \quad {{y}_{i}}\in {\mathcal{P}_{p}}, \quad 1 \le i \le 5.\]
Since
\[\big(\sum\limits_{i=1}^{5}{{{x}_{i}}}\big){{p}^{r}}+\sum\limits_{i=1}^{5}{{{y}_{i}}}=2{{p}^{r+1}},\]
we know there exists $1\le a \le 4$ such that
\[\left\{ \begin{array}{ll}
  {{x}_{1}}+\cdots +{{x}_{5}}=2p-a \\
 {{y}_{1}}+\cdots +{{y}_{5}}=ap^{r} \\
\end{array} \right..\]
By (ii) of Lemma \ref{Casolution}, we have $C_{a} \equiv 0$ (mod $p$) for $1\le a \le 4$. Hence
\begin{displaymath}
\begin{split}
   S_{5}^{(2)}({{p}^{r+1}}) & =\sum\limits_{\begin{smallmatrix}
 {{l}_{1}}+\cdots +{{l}_{5}}=2{{p}^{r+1}} \\
 {{l}_{i}}\in {\mathcal{P}_{p}},{{l}_{i}}<{{p}^{r+1}}
\end{smallmatrix}}{\frac{1}{{{l}_{1}}{{l}_{2}}\cdots {{l}_{5}}}} \\
 & =\sum\limits_{a=1}^{4}{\sum\limits_{\begin{smallmatrix}
 {{x}_{1}}+\cdots +{{x}_{5}}=2p-a \\
 0\le {{x}_{i}}<p
\end{smallmatrix}}{\sum\limits_{\begin{smallmatrix}
 {{y}_{1}}+\cdots +{{y}_{5}}=a{{p}^{r}} \\
 {{y}_{i}}\in {\mathcal{P}_{p}},{{y}_{i}}<{{p}^{r}}
\end{smallmatrix}}{\frac{1}{({{x}_{1}}{{p}^{r}}+{{y}_{1}})\cdots ({{x}_{5}}{{p}^{r}}+{{y}_{5}})}}}} \\
 & \equiv {{C}_{1}}S_{5}^{(1)}({{p}^{r}})+{{C}_{2}}S_{5}^{(2)}({{p}^{r}})+{{C}_{3}}S_{5}^{(3)}({{p}^{r}})+{{C}_{4}}S_{5}^{(4)}({{p}^{r}}) \pmod{p^{r+1}}.
\end{split}
\end{displaymath}

By (i) of Lemma \ref{rec}, we have $S_{5}^{(3)}(p^{r})\equiv -S_{5}^{(2)}(p^{r})$ (mod $p^{r}$) and $S_{5}^{(4)}(p^{r}) \equiv -S_{5}^{(1)}(p^{r})$ (mod $p^{r}$).
Hence
\begin{equation}\label{add1}
   S_{5}^{(2)}({{p}^{r+1}})  \equiv ({{C}_{1}}-{{C}_{4}})S_{5}^{(1)}({{p}^{r}})+({{C}_{2}}-{{C}_{3}})S_{5}^{(2)}({{p}^{r}}) \pmod {p^{r+1}}.
\end{equation}

By (ii) of Lemma \ref{rec} we have
\begin{equation}\label{add2}
\begin{split}
   S_{5}^{(1)}({{p}^{r+1}}) &\equiv \sum\limits_{k=1}^{4}{S_{5}^{(k)}({{p}^{r}})\binom {p-k+4}{4}} \\
 & \equiv \bigg(\binom {p+3}{4}-\binom {p}{4} \bigg)S_{5}^{(1)}({{p}^{r}})+   \bigg(\binom {p+2}{4}-\binom {p+1}{4} \bigg)S_{5}^{(2)}({{p}^{r}}) \\
 & =\frac{p({{p}^{2}}+1)}{2}S_{5}^{(1)}({{p}^{r}})+\frac{p({{p}^{2}}-1)}{6}S_{5}^{(2)}({{p}^{r}}) \pmod{{{p}^{r+1}}}.
\end{split}
\end{equation}
From (\ref{add1}) and (\ref{add2}), by induction on $r$ we can prove $S_{5}^{(1)}({{p}^{r}})\equiv S_{5}^{(2)}({{p}^{r}})\equiv 0$ (mod ${{p}^{r-1}}$). Combining this with (\ref{add1}) and (\ref{add2}), by Lemma \ref{Casolution} we deduce that
\begin{equation}\label{S512rec}
\begin{split}
&S_{5}^{(1)}({{p}^{r+1}})\equiv \frac{1}{2}pS_{5}^{(1)}({{p}^{r}})-\frac{1}{6}pS_{5}^{(2)}({{p}^{r}}) \pmod{{{p}^{r+1}}},\\
&S_{5}^{(2)}({{p}^{r+1}}) \equiv -\frac{3}{2}pS_{5}^{(1)}({{p}^{r}})+\frac{1}{2}pS_{5}^{(2)}({{p}^{r}}) \pmod {{{p}^{r+1}}}.
\end{split}
\end{equation}
From (\ref{S512rec}) we have
\begin{equation}\label{s21}
S_{5}^{(2)}({{p}^{r+1}})\equiv -3S_{5}^{(1)}({{p}^{r+1}}) \pmod{{{p}^{r+1}}}, \quad r\ge 1.
\end{equation}
Substituting this result into the first congruence of (\ref{S512rec}), we obtain
\begin{equation}\label{S5rec}
S_{5}^{(1)}({{p}^{r+1}})\equiv pS_{5}^{(1)}({{p}^{r}}) \pmod{{{p}^{r+1}}}, \quad  r\ge 2.
\end{equation}

By (\ref{zhou}), we have $S_{5}^{(1)}(p)\equiv -4!{{B}_{p-5}}$ (mod $p$). By Lemma \ref{S52modp} we have $S_{5}^{(2)}(p)\equiv 2\cdot 4!{{B}_{p-5}}$ (mod $p$). Hence from (\ref{S512rec}) we deduce that
\[S_{5}^{(1)}({{p}^{2}})\equiv \frac{1}{2}pS_{5}^{(1)}(p)-\frac{1}{6}pS_{5}^{(2)}(p)\equiv -\frac{5!}{6}p{{B}_{p-5}} \pmod{{{p}^{2}}}.\]
Now using (\ref{S5rec}) and by induction on $r$ we prove $S_{5}^{(1)}({{p}^{r}})\equiv -\frac{5!}{6}{{p}^{r-1}}{{B}_{p-5}}$ (mod ${{p}^{r}}$) for any prime $p>5$ and integer $r \ge 2$.
\end{proof}

\begin{proof}[Proof of Theorem \ref{thm2}]
Let $n=mp^{r}$, where $p$ does not divide $m$.
For any 5-tuples $({{l}_{1}},\cdots ,{{l}_{5}})$ of integers satisfying ${{l}_{1}}+\cdots +{{l}_{5}}=n$, $l_{i} \in \mathcal{P}_{p}$, $1 \le i \le 5$, we rewrite them as
\[{{l}_{i}}={{x}_{i}}{{p}^{r}}+{{y}_{i}}, \quad x_{i} \ge 0, \quad 1\le {{y}_{i}}<{{p}^{r}}, \quad {{y}_{i}}\in {\mathcal{P}_{p}}, \quad 1 \le i \le 5.\]
Since
\[\big(\sum\limits_{i=1}^{5}{{{x}_{i}}}\big){{p}^{r}}+\sum\limits_{i=1}^{5}{{{y}_{i}}}=m{{p}^{r}},\]
we know there exists $1\le a \le 4$ such that
\[\left\{ \begin{array}{ll}
  {{x}_{1}}+\cdots +{{x}_{5}}=m-a \\
 {{y}_{1}}+\cdots +{{y}_{5}}=ap^{r} \\
\end{array} \right..\]
For $1 \le a \le 4$, the equation $x_{1}+x_{2}+x_{3}+x_{4}+x_{5}=m-a$ has $\binom {m-a+4}{4}$ solutions $(x_{1},x_{2},x_{3},x_{4},x_{5})$ of nonnegative integers. Hence
\begin{equation}\label{thm2rec}
\begin{split}
  \sum\limits_{\begin{smallmatrix}
 {{l}_{1}}+{{l}_{2}}+\cdots +{{l}_{5}}=n \\
 {{l}_{1}},\cdots ,{{l}_{5}}\in {\mathcal{P}_{p}}
\end{smallmatrix}}{\frac{1}{{{l}_{1}}{{l}_{2}}{{l}_{3}}{{l}_{4}}{{l}_{5}}}}  & =\sum\limits_{\begin{smallmatrix}
 {{l}_{1}}+\cdots +{{l}_{5}}=m{{p}^{r}} \\
  {{l}_{1}},\cdots ,{{l}_{5}}\in {\mathcal{P}_{p}}
\end{smallmatrix}}{\frac{1}{{{l}_{1}}{{l}_{2}}\cdots {{l}_{5}}}} \\
 & =\sum\limits_{a=1}^{4}{\sum\limits_{\begin{smallmatrix}
 {{x}_{1}}+\cdots +{{x}_{5}}=m-a \\
 0\le {{x}_{i}}<p
\end{smallmatrix}}{\sum\limits_{\begin{smallmatrix}
 {{y}_{1}}+\cdots +{{y}_{5}}=a{{p}^{r}} \\
 {{y}_{i}}\in {\mathcal{P}_{p}},{{y}_{i}}<{{p}^{r}}
\end{smallmatrix}}{\frac{1}{({{x}_{1}}{{p}^{r}}+{{y}_{1}})\cdots ({{x}_{5}}{{p}^{r}}+{{y}_{5}})}}}} \\
 & \equiv \binom{m+3}{4}S_{5}^{(1)}({{p}^{r}})+\binom{m+2}{4}S_{5}^{(2)}({{p}^{r}}) \\
 & \quad +\binom{m+1}{4}S_{5}^{(3)}({{p}^{r}})+\binom{m}{4}S_{5}^{(4)}({{p}^{r}}) \pmod{p^{r}}.
\end{split}
\end{equation}
According to $r=1$ or $r \ge 2$, we split our proof into two cases.

(i) If $r=1$, then from (\ref{zhou}) we know $S_{5}^{(1)}(p)\equiv -4!B_{p-5}$ (mod $p$). By Lemma \ref{S52modp}, we deduce that $S_{5}^{(2)}(p) \equiv -2 S_{5}^{(1)}(p)$ (mod $p$). By (i) of Lemma \ref{rec}, we know $S_{5}^{(3)}(p) \equiv 2S_{5}^{(1)}(p)$ (mod $p$) and $S_{5}^{(4)}(p) \equiv -S_{5}^{(1)}(p)$ (mod $p$). Hence from (\ref{thm2rec}) we have
\begin{displaymath}
\begin{split}
 \sum\limits_{\begin{smallmatrix}
 {{l}_{1}}+{{l}_{2}}+\cdots +{{l}_{5}}=n \\
 {{l}_{1}},\cdots ,{{l}_{5}}\in {\mathcal{P}_{p}}
\end{smallmatrix}}{\frac{1}{{{l}_{1}}{{l}_{2}}{{l}_{3}}{{l}_{4}}{{l}_{5}}}}  & \equiv \bigg(\binom{m+3}{4}-2\binom{m+2}{4}+2\binom{m+1}{4}-\binom{m}{4}\bigg)S_{5}^{(1)}(p) \\
 & \equiv \frac{1}{6}(5m+m^{3}) S_{5}^{(1)}(p)  \pmod{p}.
 \end{split}
\end{displaymath}
Since $S_{5}^{(1)}(p)\equiv -4!B_{p-5}$ (mod $p$) and $m=\frac{n}{p}$, we complete the proof of (i).

(ii) If $r \ge 2$, then from (\ref{s21}) we deduce that for any integer $r \ge 2$, $S_{5}^{(2)}({{p}^{r}})\equiv -3S_{5}^{(1)}(p^{r})$ (mod ${{p}^{r}}$). By (i) of Lemma \ref{rec}, we obtain
\[ S_{5}^{(3)}({{p}^{r}}) \equiv -S_{5}^{(2)}({{p}^{r}}) \equiv 3S_{5}^{(1)}(p^{r}) \pmod{{{p}^{r}}}, \quad  S_{5}^{(4)}({{p}^{r}})  \equiv -S_{5}^{(1)}({{p}^{r}})  \pmod{{p}^{r}}.\]
Hence from (\ref{thm2rec}) we obtain
\begin{displaymath}
\begin{split}
 \sum\limits_{\begin{smallmatrix}
 {{l}_{1}}+{{l}_{2}}+\cdots +{{l}_{5}}=n \\
 {{l}_{1}},\cdots ,{{l}_{5}}\in {\mathcal{P}_{p}}
\end{smallmatrix}}{\frac{1}{{{l}_{1}}{{l}_{2}}{{l}_{3}}{{l}_{4}}{{l}_{5}}}}  & \equiv \bigg(\binom{m+3}{4}-3\binom{m+2}{4}+3\binom{m+1}{4}-\binom{m}{4}\bigg)S_{5}^{(1)}(p^{r}) \\
 & \equiv m S_{5}^{(1)}(p^{r})  \pmod{p^{r}}.
 \end{split}
\end{displaymath}
By Theorem \ref{main}, we have $S_{5}^{(1)}(p^{r}) \equiv -\frac{5!}{6}p^{r-1}B_{p-5}$ (mod $p^{r}$). This completes the proof of (ii).
\end{proof}
As we mentioned earlier, naturally we can ask the following question: Can we find two arithmetical functions $a(n)$ and $b(n)$ such that \\
(i) for any odd integer $n \ge 7$ and prime $p>n$,
\[\sum\limits_{\begin{smallmatrix}
 {{l}_{1}}+{{l}_{2}}+\cdots +{{l}_{n}}={{p}^{r}} \\
 {{l}_{1}},\cdots ,{{l}_{n}}\in {\mathcal{P}_{p}}
\end{smallmatrix}}{\frac{1}{{{l}_{1}}{{l}_{2}}\cdots {{l}_{n}}}}\equiv a(n){{p}^{r-1}}{{B}_{p-n}} \pmod{{{p}^{r}}};\]
(ii) for any even integer $n \ge 6$ and prime $p>n$,
\[\sum\limits_{\begin{smallmatrix}
 {{l}_{1}}+{{l}_{2}}+\cdots +{{l}_{n}}={{p}^{r}} \\
 {{l}_{1}},\cdots ,{{l}_{n}}\in {\mathcal{P}_{p}}
\end{smallmatrix}}{\frac{1}{{{l}_{1}}{{l}_{2}}\cdots {{l}_{n}}}}\equiv b(n){{p}^{r}}{{B}_{p-n-1}} \pmod {{{p}^{r+1}}}.\]
At this stage, we are not able to answer this question. We believe such $a(n)$ and $b(n)$ exist but to solve this problem one may need to develop some new ideas or methods.

% -----------------------------------------------------------

\end{document}